\documentclass[11pt,draft]{amsart}
\DeclareMathAlphabet{\EuFrak}{U}{euf}{m}{n}
\DeclareMathAlphabet{\EuScript}{U}{eus}{m}{n}
\usepackage{amsmath}
\usepackage{amsfonts,amssymb,latexsym}
\usepackage{mathrsfs}
\usepackage{amscd}
\usepackage{amsthm}
\newtheorem{theorem}{\rmfamily\bfseries{Theorem}}[section] 
\newtheorem{corollary}[theorem]{\rmfamily\bfseries{Corollary}} 

\newtheorem{lemma}[theorem]{\rmfamily\bfseries{Lemma}} 
\newtheorem{proposition}[theorem]{\rmfamily\bfseries{Proposition}}

\newtheorem{definition}[theorem]{\rmfamily\bfseries{Definition}}

\theoremstyle{remark}

\newtheorem{remark}{Remark}

\usepackage{graphics}
\numberwithin{equation}{section}
\def\Bbb{\mathbb}
\newcommand{\mona}{\mbox{\LARGE\itshape a}}
\def\logo{\raisebox{-10.5\p@}{\hb@xt@85\p@{\includegraphics{gft.eps}\hfil}}}

\def\un{1\kern-3pt \rm I}

\newcommand{\oN}{{\mathbb N}}
\newcommand{\oR}{{\mathbb R}}

\newcommand{\oC}{{\mathbb C}}

\newcommand{\supp}{\mathrm{supp}}
   
\begin{document}
{\hfill
\parbox{50mm}{{\sf CEFT-SFM-DHTF07/2}} \vspace{8mm}}

\title[A Uniqueness Theorem]
      {{A Uniqueness Theorem and Its Application \\ to Field-Theoretical
           Models with a Fundamental Length}}

\author{Daniel H.T. Franco}
\address{Universidade Federal de Vi\c cosa \\
        Departamento de F\'\i sica, Avenida Peter Henry Rolfs s/n \\  
        Campus Universit\'ario, Vi\c cosa, MG, Brasil, CEP: 36570-000.}
\email{dhtf@terra.com.br}

\date{\today}
\thanks{This work is supported by the Funda\c c\~ao de Amparo \`a Pesquisa do
        Estado de Minas Gerais (FAPEMIG) agency, grant CEX00012/07. Also at
        Centro de Estudos de F\'\i sica Te\'orica, Setor de F\'\i sica--Matem\'atica,
        Belo Horizonte, MG, Brasil.}

\begin{abstract}
It is shown that if a distribution $V$ of exponential growth has support
in a proper convex cone and its Fourier transform is carried by a closed
cone different from whole space, then $V=0$. The application of this result
to a {\em quasi-local} quantum field theory (where the fields are
localizable only in regions greater than a certain scale of nonlocality) is
contemplated. In particular, we show that a number of physically important
predictions of {\em local} quantum field theory also hold in a quantum field theory
with a fundamental length, as indicated from string theory. 
\end{abstract}

\maketitle

\section{Introduction}
In Ref.~\cite{Solo1,Solo2} Soloviev showed that if a distribution $u \in {\mathscr D}^\prime$
has support in a proper convex cone and its Fourier transform, an analytic
functional $v$ belonging to the space $Z^\prime$ of ultradistributions of Gel'fand and
Shilov, is carried by a closed cone $C$ different from
the whole space, then $u \equiv 0$. (In~\cite{Solo1,Solo2} Soloviev uses the notation
$S^{\prime 0}$ in place of $Z^\prime$ in order to stress that this is the
smallest space among the Gel'fand-Shilov~\cite{GelShi} spaces $S^{\prime \beta}$,
$0 \leq \beta < 1$, traditionally adopted in nonlocal quantum field theory).
His proof is based on the notion of the analytic wavefront
set of a distribution and makes possible to deal nonlocal quantum fields. In this paper,
we show that a similar uniqueness theorem also holds for the space of distributions of
exponential growth $V \in H^\prime$, since the latter is embedded in the space of
distributions. Thus, we can use the general facts from distribution theory to analyse the
analytic wavefront set of a distribution of exponential growth. It is known that the Fourier
transform is a topological isomorphism between elements in $H^\prime$ and elements in the
space of tempered ultrahyperfunctions ${\mathfrak H}^\prime$. The tempered ultrahyperfunctions,
ori\-gi\-nally called {\em tempered ultradistributions}, has been studied by many
authors~\cite{Tiao1}-\cite{Daniel2} and represents a natural generalization of the notion of
hyperfunctions on $\oR^n$, but are {\it non-localizable}. We shall show that if a distribution
$V$ of exponential growth has support in a proper convex cone and its Fourier transform, an
analytic functional $U \in {\mathfrak H}^\prime$, is carried by a closed cone different from
whole space, then $V \equiv 0$. 

The plan of the paper is as follows. In Section \ref{Sec2}, we introduce the notation and
definitions used here. In Section \ref{Sec3}, we shall collect some facts of the theory on tempered
ultrahyperfunctions. There we define the space of tempered ultrahyperfunctions corresponding
to a proper open convex cone. Properties of analytic functionals in ${\mathfrak H}^\prime$ with
real unbounded carriers are investigated in the Section \ref{Sec4}. Section \ref{Sec5} is devoted
to the proof of uniqueness theorem. We note that this result is of importance in the construction
and study of {\em quasilocal} quantum field theories (where the fields are localizable only in
regions greater than a certain scale of nonlocality). For this reason, in Section \ref{Sec6}, as an
application of the uniqueness theorem, we give also a proof of the validity of some important theorems
in quantum field theory, namely the proofs of the CPT theorem (which is the basis of particle and
anti-particle symmetry) and the theorem on the Spin-Statistics connection (which is the basis for
the stability of matter in the Nature) in the setting of a quantum field theory with a fundamental
length. This section is meant for mathematicians who also want to become acquainted with the
applications of tempered ultrahyperfunctions in physics, as well as for physicists who are interested
in tempered ultrahyperfunctions as part of mathematical and theoretical physics.

\section{Notation and Definitions}
\label{Sec2}
The following multi-index notation is used without further explanation. Let
$\oR^n$ (resp. $\oC^n=\oR^n+i\oR^n$) be the real (resp. complex) $n$-space whose generic
points are denoted by $x=(x_1,\ldots,x_n)$ (resp. $z=(z_1,\ldots,z_n)$), such that
$x+y=(x_1+y_1,\ldots,x_n+y_n)$, $\lambda x=(\lambda x_1,\ldots,\lambda x_n)$,
$x \geq 0$ means $x_1 \geq 0,\ldots,x_n \geq 0$, $\langle x,y \rangle=x_1y_1+\cdots+x_ny_n$
and $|x|=|x_1|+\cdots+|x_n|$. Moreover, we define
$\alpha=(\alpha_1,\ldots,\alpha_n) \in \oN^n_o$, where $\oN_o$ is the set
of non-negative integers, such that the length of $\alpha$ is the corresponding
$\ell^1$-norm $|\alpha|=\alpha_1+\cdots +\alpha_n$, $\alpha+\beta$ denotes
$(\alpha_1+\beta_1,\ldots,\alpha_n+\beta_n)$, $\alpha \geq \beta$ means
$(\alpha_1 \geq \beta_1,\ldots,\alpha_n \geq \beta_n)$, $\alpha!=
\alpha_1! \cdots \alpha_n!$, $x^\alpha=x_1^{\alpha_1}\ldots x_n^{\alpha_n}$,
and
\[
D^\alpha \varphi(x)=\frac{\partial^{|\alpha|}\varphi(x_1,\ldots,x_n)}
{\partial x_1^{\alpha_1}\partial x_2^{\alpha_1}\ldots\partial x_n^{\alpha_n}}\,\,.
\]
Let $\Omega$ be a set in $\oR^n$. Then we denote by $\Omega^\circ$ the interior
of $\Omega$ and by $\overline{\Omega}$ the closure of $\Omega$. For $r > 0$, we
denote by $B(x_o;r)=\bigl\{x \in \oR^n \mid |x-x_o| < r\bigr\}$ a open ball
and by $B[x_o;r]=\bigl\{x \in \oR^n \mid |x-x_o| \leq r\bigr\}$ a closed ball,
with center at point $x_o$ and of radius $r=(r_1,\ldots,r_n)$, respectively.

We consider two $n$-dimensional spaces -- $x$-space and $\xi$-space -- with the
Fourier transform defined
\[
\widehat{f}(\xi)={\mathscr F}[f(x)](\xi)=
\int_{\oR^n} f(x)e^{i \langle \xi,x \rangle} d^nx\,\,,
\]
while the Fourier inversion formula is
\[
f(x)={\mathscr F}^{-1}[\widehat{f}(\xi)](x)= \frac{1}{(2\pi)^n}
\int_{\oR^n} \widehat{f}(\xi)e^{-i \langle \xi,x \rangle} d^n\xi\,\,.
\]
The variable $\xi$ will always be taken real while $x$ will also be
complexified -- when it is complex, it will be noted $z=x+iy$. The
above formulas, in which we employ the symbolic ``function notation,''
are to be understood in the sense of distribution theory.

We now remind some terminology and simple facts concerning cones. An open
set $C \subset \oR^n$ is called a cone if $\oR_+\!\cdot C \subset C$. A cone
$C$ is an open connected cone if $C$ is an open connected set. Moreover, $C$ is
called convex if $C+C \subset C$ and {\it proper} if it contains no any straight
line. A cone $C^\prime$ is called compact in $C$ --
we write $C^\prime \Subset C$ -- if the projection ${\sf pr}{\overline C^{\,\prime}}
\overset{\text{def}}{=}{\overline C^{\,\prime}} \cap S^{n-1} \subset
{\sf pr}C\overset{\text{def}}{=}C \cap S^{n-1}$, where $S^{n-1}$ is the unit sphere
in $\oR^n$. Being given a cone $C$ in $y$-space, we associate with $C$ a closed convex
cone $C^*$ in $\xi$-space which is the set $C^*=\bigl\{\xi \in \oR^n \mid \langle \xi,y
\rangle \geq 0,\forall\,\,y \in C \bigr\}$. The cone $C^*$ is called the {\em dual cone}
of $C$. In the sequel, it will be sufficient to assume for our purposes that the open
connected cone $C$ in $\oR^n$ is an open convex cone with vertex at the origin
and proper. By $T(C)$ we will denote the set $\oR^n+iC \subset \oC^n$.
If $C$ is open and connected, $T(C)$ is called the tubular radial domain in
$\oC^n$, while if $C$ is only open $T(C)$ is referred to as a tubular cone. In the
former case we say that $f(z)$ has a boundary value $U=BV(f(z))$ in ${\mathfrak H}^\prime$
as $y \rightarrow 0$, $y \in C$ or $y \in C^\prime \Subset C$, respectively, if
for all $\psi \in {\mathfrak H}$ the limit
\[
\langle U, \psi \rangle=\lim_{{\substack{y \rightarrow 0 \\
y \in C~{\rm or}~C^\prime}}} \int_{\oR^n} f(x+iy)\psi(x) d^nx\,\,,
\]
exists. We will deal with tubes defined as the set of all points $z \in \oC^n$
such that
\[
T(C)=\Bigl\{x+iy \in \oC^n \mid x \in \oR^n, y \in C, |y| < \delta \Bigr\}\,\,,
\]
where $\delta > 0$ is an arbitrary number.

\section{Tempered Ultrahyperfunctions}
\label{Sec3}
We shall introduce briefly here some definitions and basic properties of the
tempered ultrahyperfunction space of Sebasti\~ao e Silva~\cite{Tiao1,Tiao2}
and Hasumi~\cite{Hasumi} (we indicate the Refs. for more details).
To begin with, we shall consider the function
\[
h_{K}(\xi)=\sup_{x \in K} \langle \xi,x \rangle\,\,,
\quad \xi \in \oR^n\,\,,
\]
where $K$ is a compact set in $\oR^n$. One calls $h_{K}(\xi)$ the {\it supporting
function} of $K$. We note that $h_{K}(\xi) < \infty$ for every $\xi \in \oR^n$ since
$K$ is bounded. For sets $K=\bigl[-k,k\bigr]^n$, $0 < k < \infty$, the supporting
function $h_{K}(\xi)$ can be easily determined:
\[
h_{K}(\xi)=\sup_{x \in K} \langle \xi,x \rangle =
k|\xi|\,\,,\quad \xi \in \oR^n\,\,,\quad |\xi|=\sum_{i=1}^n|\xi_i|\,\,.
\]

Let $K$ be a convex compact subset of $\oR^n$,
then $H_b(\oR^n;K)$ ($b$ stands for bounded) defines the space of all
functions $\in C^\infty(\oR^n)$ such that $e^{h_K(\xi)}D^\alpha\!f(\xi)$
is bounded in $\oR^n$ for any multi-index $\alpha$. One defines in
$H_b(\oR^n;K)$ seminorms
\begin{equation}
\|\varphi\|_{K,N}=\sup_{{\substack{\xi \in \oR^n \\ \alpha \leq N}}}
\bigl\{e^{h_K(\xi)}|D^\alpha f(\xi)|\bigr\} < \infty\,\,,
\quad N \in \oN \,\,.
\label{snorma2}
\end{equation}

If $K_1 \subset K_2$ are two compact convex sets, then
$h_{K_1}(\xi) \leq h_{K_2}(\xi)$, and thus the canonical
injection $H_b(\oR^n;K_2) \hookrightarrow H_b(\oR^n;K_1)$
is continuous. Let $O$ be a convex open set of $\oR^n$.
To define the topology of $H(\oR^n;O)$ it suffices to let $K$ range
over an increasing sequence of convex compact subsets $K_1,K_2,\ldots$
contained in $O$ such that for each $i=1,2,\ldots$,
$K_i \subset K_{i+1}^\circ$ and ${O}=\bigcup_{i=1}^\infty K_i$.
Then the space $H(\oR^n;O)$ is the projective limit of the
spaces $H_b(\oR^n;K)$ according to restriction mappings
above, {\em i.e.}
\begin{equation}
H(\oR^n;O)=\underset{K \subset {O}}{\lim {\rm proj}}\,\,
H_b(\oR^n;K)\,\,,
\label{limproj2}
\end{equation} 
where $K$ runs through the convex compact sets contained in $O$. Any
$C^\infty$ function of exponential growth is a multiplier in $H(\oR^n;O)$.

\begin{theorem}[\cite{Hasumi,Mari1,BruNa1}]
The space ${\mathscr D}({\oR^n})$ of all $C^\infty$-functions
on $\oR^n$ with compact support is dense in $H(\oR^n;K)$ and $H(\oR^n;O)$.
Moreover, the space $H(\oR^n;\oR^n)$ is dense in $H(\oR^n;O)$ and in
$H(\oR^n;K)$, and $H(\oR^m;\oR^m) \otimes H(\oR^n;\oR^n)$ is dense in
$H(\oR^{m+n};\oR^{m+n})$.
\label{theoINJ}
\end{theorem}

From Theorem \ref{theoINJ} we have the following injections~\cite{Mari1}:
\[
H^\prime(\oR^n;K) \hookrightarrow H^\prime(\oR^n;\oR^n)
\hookrightarrow {\mathscr D}^\prime(\oR^n)\,\,,
\]
and
\[
H^\prime(\oR^n;O) \hookrightarrow H^\prime(\oR^n;\oR^n)
\hookrightarrow {\mathscr D}^\prime(\oR^n)\,\,.
\]

\begin{definition}
The dual space $H^\prime(\oR^n;O)$ of $H(\oR^n;O)$ is the space of
distributions of exponential growth.
\end{definition}

A distribution $V \in H^\prime(\oR^n;O)$ may be expressed as a finite
order deri\-va\-ti\-ve of a continuous function of exponential growth
\[
V=D^\gamma_\xi[e^{h_K(\xi)}g(\xi)]\,\,,
\]
where $g(\xi)$ is a bounded continuous function. For $V \in
H^\prime(\oR^n;O)$ the follo\-wing result is known:

\begin{lemma}[\cite{Mari1}]
A distribution $V \in {\mathscr D}^\prime(\oR^n)$ belongs to $H^\prime(\oR^n;O)$
if and only if there exists a multi-index $\gamma$, a convex compact set $K \subset O$
and a bounded continuous function $g(\xi)$ such that
\[
V=D^\gamma_\xi[e^{h_K(\xi)}g(\xi)]\,\,.
\]
\label{lemmaMari}
\end{lemma}

In the space $\oC^n$ of $n$ complex variables $z_i=x_i+iy_i$,
$1 \leq i \leq n$, we denote by $T(\Omega)=\oR^n+i\Omega \subset \oC^n$
the tubular set of all points $z$, such that $y_i={\text{Im}}\,z_i$ belongs
to the domain $\Omega$, {\em i.e.}, $\Omega$ is a connected open set in $\oR^n$
called the basis of the tube $T(\Omega)$. Let $K$ be a convex compact
subset of $\oR^n$, then ${\mathfrak H}_b(T(K))$ defines
the space of all continuous functions $\varphi$ on $T(K)$ which are holomorphic
in the interior $T(K^\circ)$ of $T(K)$ such that the estimate
\begin{equation}
|\varphi(z)| \leq {\boldsymbol{\sf M}}_{_{T(K),N}}(\varphi) (1+|z|)^{-N}
\label{est}
\end{equation}
is valid. The best possible constants in (\ref{est}) are given by a family of seminorms in
${\mathfrak H}_b(T(K))$
\begin{equation}
\|\varphi\|_{T(K),N}=\inf\Bigl\{{\boldsymbol{\sf M}}_{_{T(K),N}}(\varphi) \mid
\sup_{z \in T(K)} \bigl\{(1+|z|)^N|\varphi(z)|\bigr\} < \infty, N \in \oN \Bigr\}\,\,.
\label{snorma1}
\end{equation}

If $K_1 \subset K_2$ are two convex compact sets, we have that the canonical injection
\begin{equation}
{\mathfrak H}_b(T(K_2)) \hookrightarrow {\mathfrak H}_b(T(K_1)\,\,,
\label{canoinj}
\end{equation}
is continuous.

Let $K$ be a convex compact set in $\oR^n$. Then the space ${\mathfrak H}(T(K))$
is characterized as a inductive limit
\begin{equation}
{\mathfrak H}(T(K))=\underset{K_1 \supset K}{\lim {\rm ind}}\,\,
{\mathfrak H}_b(T(K_1))\,\,,
\label{limind1}
\end{equation}
where $K_1$ runs through the convex compact sets such that $K$ is contained
in the interior of $K_1$ and the inductive limit is taken following the restriction
mappings (\ref{canoinj}).

Given that the spaces ${\mathfrak H}_b(T(K_i))$ are Fr\'echet spaces, with topology
defined by the seminorms (\ref{snorma1}), the space ${\mathfrak H}(T({O}))$ is
characterized as a projective limit of Fr\'echet spaces: 
\begin{equation}
{\mathfrak H}(T({O}))=\underset{K \subset {O}}{\lim {\rm proj}}\,\,
{\mathfrak H}_b(T(K))\,\,,
\label{limproj1}
\end{equation}
where $K$ runs through the convex compact sets contained in $O$ and
the projective limit is taken following the restriction mappings above. Any
$C^\infty$ function which can be extended to be an entire function of polynomial
growth, that is, slow growth, is a multiplier in ${\mathfrak H}(T({O}))$.

For any element $U \in {\mathfrak H}^\prime$, its Fourier transform is
defined to be a distribution $V$ of exponential growth, such that the
Parseval-type relation
\begin{equation}
\langle V,\varphi \rangle=\langle U,\psi \rangle\,\,,\quad
\varphi \in H\,\,,\,\,\psi={\mathscr F}[\varphi] \in {\mathfrak H}\,\,,
\label{PRel1}
\end{equation}
holds. In the same way, the inverse Fourier transform of a distribution $V$ of
exponential growth is defined by the relation 
\begin{equation}
\langle U,\psi \rangle=\langle V,\varphi \rangle\,\,,\quad
\psi \in {\mathfrak H}\,\,,\,\,\varphi={\mathscr F}^{-1}[\psi] \in H\,\,.
\label{PRel2}
\end{equation}

It follows from the Fourier transform and Theorem \ref{theoINJ} the

\begin{theorem}[\cite{Mari1,BruNa1}]
The space ${\mathfrak H}(T(\oR^n))$ is dense in ${\mathfrak H}(T(O))$ and
in ${\mathfrak H}(T(K))$, and the space ${\mathfrak H}(T(\oR^{m+n}))$ is dense
in ${\mathfrak H}(T(O))$.
\label{theoINJE}
\end{theorem}

\begin{proposition}[\cite{Mari1}]
If $f \in H(\oR^n;O)$, the Fourier transform of $f$ belongs
to the space ${\mathfrak H}(T(O))$, for any open convex
non-empty set $O \subset \oR^n$. By the dual Fourier transform
$H^\prime(\oR^n;O)$ is topologically isomorphic with the space
${\mathfrak H}^\prime(T(-O))$.
\label{Propo1}
\end{proposition}

Let us now recall very briefly the basic definition of tempered ultrahyperfunctions.
These are defined as elements of a certain subspace of $Z^\prime$ of
ultradistributions of Gel'fand and Shilov which admit representations in terms of
analytic functions on the complement of some closed horizontal strip of the complex
space, and having polynomial growth on the complement of an open neighborhood of
that strip.

Let $\boldsymbol{{\mathscr H}_\omega}$ be the space of all
functions $f(z)$ such that ({\it i}) $f(z)$ is analytic for $\{z \in \oC^n \mid
|{\rm Im}\,z_1| > p, |{\rm Im}\,z_2| > p,\ldots,|{\rm Im}\,z_n| > p\}$,
({\it ii}) $f(z)/z^p$ is bounded continuous  in
$\{z\in \oC^n \mid |{\rm Im}\,z_1| \geqq p,|{\rm Im}\,z_2| \geqq p,
\ldots,|{\rm Im}\,z_n| \geqq p\}$, where $p=0,1,2,\ldots$ depends on $f(z)$
and ({\it iii}) $f(z)$ is bounded by a power of $z$, $|f(z)|\leq
{\boldsymbol{\sf M}}(1+|z|)^N$, where ${\boldsymbol{\sf M}}$ and $N$ depend on
$f(z)$. Define the {\em kernel} of the mapping $f:{\mathfrak H}(T(\oR^n))
\rightarrow \oC$ by $\boldsymbol{\Pi}$, as the set of all $z$-dependent
pseudo-polynomials, $z\in \oC^n$ (a pseudo-polynomial is a
function of $z$ of the form $\sum_s z_j^s G(z_1,...,z_{j-1},z_{j+1},...,z_n)$,
with $G(z_1,...,z_{j-1},z_{j+1},...,z_n) \in \boldsymbol{{\mathscr H}_\omega}$).
Then, $f(z) \in \boldsymbol{{\mathscr H}_\omega}$ belongs to the kernel
$\boldsymbol{\Pi}$ if and only if $\langle f(z),\psi(x) \rangle=0$,
with $\psi(x) \in {\mathfrak H}(T(\oR^n))$ and $x={\rm Re}\,z$.
Consider the quotient space ${\mathscr U}=\boldsymbol{{\mathscr H}_\omega}
/\boldsymbol{\Pi}$. The set ${\mathscr U}$ is the space of tempered
ultrahyperfunctions. Thus, we have the

\begin{definition}
The space of tempered ultrahyperfunctions, denoted by ${\mathscr U}(\oR^n)$,
is the space of continuous linear functionals defined on ${\mathfrak H}(T(\oR^n))$. 
\label{UHF}
\end{definition}

In the sequel we will put ${\mathfrak H}={\mathfrak H}(\oC^n)={\mathfrak H}(T(\oR^n))$
and the dual space of ${\mathfrak H}$ will be denoted by ${\mathfrak H}^\prime$.

\begin{theorem}[Hasumi~\cite{Hasumi}, Proposition 5] The space of
tempered ultrahyperfunctions ${\mathscr U}$ is algebraically isomorphic
to the space of generalized functions ${\mathfrak H}^\prime$.
\label{HasumiTheo}
\end{theorem}

\subsection{Tempered Ultrahyperfunctions Corresponding to a Proper Convex
Cone}
\label{Sec31}
Let $C$ be a proper open convex cone, and let $C^\prime \Subset C$.
Let $B[0;r]$ denote a {\bf closed} ball of the
origin in $\oR^n$ of radius $r$, where $r$ is an arbitrary positive
real number. Denote $T(C^\prime;r)=\oR^n+i\bigl(C^\prime \setminus
\bigl(C^\prime \cap B[0;r]\bigr)\bigr)$.
We are going to introduce a space of holomorphic functions
which satisfy certain estimate according to Carmichael~\cite{Carmi1}.
We want to consider the space consisting of holomorphic functions $f(z)$
such that
\begin{equation}
\bigl|f(z)\bigr|\leq {\boldsymbol{\sf M}}(C^\prime)(1+|z|)^N e^{h_{C^*}(y)}
\,\,,\quad z \in T(C^\prime;r)\,\,,
\label{Estimate1} 
\end{equation}
where $h_{C^*}(y)=\sup_{\xi \in C^*}\langle \xi,y \rangle$ is the supporting
function of $C^*$, ${\boldsymbol{\sf M}}(C^\prime)$ is a constant that depends
on an arbitrary compact cone $C^\prime$ and $N$ is a non-negative real number.
The set of all functions $f(z)$ which are holomorphic in $T(C^\prime;r)$ and
satisfy the estimate (\ref{Estimate1}) will be denoted by $\boldsymbol{{\mathscr H}^o_c}$.

\begin{remark}
The space of functions $\boldsymbol{{\mathscr H}^o_c}$ constitutes a generalization
of the space ${\mathfrak A}_{_\omega}^i$ of Sebati\~ao e Silva~\cite{Tiao1} and the
space $\mona_{_\omega}$ of Hasumi~\cite{Hasumi} to arbitrary tubular radial domains
in $\oC^n$.
\end{remark}

\begin{lemma}[\cite{Carmi1,DanHenri}]
Let $C$ be an open convex cone, and let $C^\prime \Subset C$. Let
$h(\xi)=e^{k|\xi|}g(\xi)$, $\xi \in \oR^n$, be a function with support
in $C^*$, where $g(\xi)$ is a bounded continuous function on $\oR^n$.
Let $y$ be an arbitrary but fixed point of $\bigl(C^\prime \setminus
\bigl(C^\prime \cap B[0;r]\bigr)\bigr)$. Then $e^{-\langle \xi,y \rangle}
h(\xi) \in L^2$, as a function of $\xi \in \oR^n$.
\label{lemma0}
\end{lemma}

\begin{definition}
We denote by $H^\prime_{C^*}(\oR^n;O)$ the subspace of $H^\prime(\oR^n;O)$
of distributions of exponential growth with support in the cone $C^*$:
\begin{equation}
H^\prime_{C^*}(\oR^n;O)=\Bigl\{V \in H^\prime(\oR^n;O) \mid
\supp(V) \subseteq C^* \Bigr\}\,\,. 
\label{Suporte} 
\end{equation}
\label{Def1}
\end{definition}

\begin{lemma}[\cite{Carmi1,DanHenri}]
Let $C$ be an open convex cone, and let $C^\prime \Subset C$.
Let $V=D^\gamma_\xi[e^{h_K(\xi)}g(\xi)]$, where $g(\xi)$ is a bounded continuous
function on $\oR^n$ and $h_K(\xi)=k|\xi|$ for a convex compact set
$K=\bigl[-k,k\bigr]^n$. Let $V \in H^\prime_{C^*}(\oR^n;O)$. Then $f(z)=(2\pi)^{-n}
\bigl\langle V,e^{-i\langle \xi,z \rangle}\bigr\rangle$ is an element of
$\boldsymbol{{\mathscr H}^o_c}$.
\label{lemma1}
\end{lemma}

We now shall define the main space of holomorphic functions with which this paper
is concerned. Let $C$ be a proper open convex cone, and let $C^\prime \Subset C$.
Let $B(0;r)$ denote an {\bf open} ball of the origin in $\oR^n$ of radius
$r$, where $r$ is an arbitrary positive real number. Denote $T(C^\prime;r)=
\oR^n+i\bigl(C^\prime \setminus \bigl(C^\prime \cap B(0;r)\bigr)\bigr)$. Throughout
this section, we consider functions $f(z)$ which are holomorphic in
$T(C^\prime)=\oR^n+iC^\prime$ and which satisfy the estimate (\ref{Estimate1}),
with $B[0;r]$ replaced by $B(0;r)$. We denote this space by
$\boldsymbol{{\mathscr H}^{*\,o}_c}$. We note that $\boldsymbol{{\mathscr H}^{*\,o}_c}
\subset \boldsymbol{{\mathscr H}^{o}_c}$ for any open convex cone $C$. Put
${\mathscr U}_c=\boldsymbol{{\mathscr H}^{*\,o}_c}/\boldsymbol{\Pi}$, that is,
${\mathscr U}_c$ is the quotient space of $\boldsymbol{{\mathscr H}^{*\,o}_c}$
by set of pseudo-polynomials $\boldsymbol{\Pi}$.

\begin{definition}
The set ${\mathscr U}_c$ is the space of tempered ultrahyperfunctions corresponding
to a proper open convex cone $C \subset \oR^n$.
\end{definition}

The following theorem shows that functions in $\boldsymbol{{\mathscr H}^{*\,o}_c}$
have distributional boun\-da\-ry values in ${\mathfrak H}^\prime(T(O))$. Further, it shows
that functions in $\boldsymbol{{\mathscr H}^{*\,o}_c}$ satisfy a strong boundedness
property in ${\mathfrak H}^\prime(T(O))$.

\begin{theorem}[\cite{Daniel1}]
Let $C$ be an open convex cone, and let $C^\prime \Subset C$.
Let $V=D^\gamma_\xi[e^{h_K(\xi)}g(\xi)]$, where $g(\xi)$ is a bounded continuous
function on $\oR^n$ and $h_K(\xi)=k|\xi|$ for a convex compact set
$K=\bigl[-k,k\bigr]^n$. Let $V \in H^\prime_{C^*}(\oR^n;O)$. Then

\,\,\,$(i)\quad f(z)=(2\pi)^{-n}\bigl\langle V,e^{-i\langle \xi,z \rangle}\bigr\rangle$
is an element of $\boldsymbol{{\mathscr H}^{*\,o}_c}$,

\,\,\,$(ii)\quad \bigl\{f(z) \mid y={\rm Im}\,z \in C^\prime \Subset C, |y| \leq Q\bigr\}$
is a strongly bounded set in ${\mathfrak H}^\prime(T(O))$, where $Q$ is an arbitrarily but
fixed positive real number,

\,\,\,$(iii)\quad f(z) \rightarrow {\mathscr F}^{-1}[V] \in {\mathfrak H}^\prime(T(O))$ in
the strong (and weak) topology of ${\mathfrak H}^\prime(T(O))$ as $y={\rm Im}\,z \rightarrow
0$, $y \in C^\prime \Subset C$.
\label{theorem1}
\end{theorem}

The functions $f(z) \in \boldsymbol{{\mathscr H}^{*\,o}_c}$ can be recovered as
the (inverse) Fourier-Laplace transform of the constructed distribution $V \in H^\prime_{C^*}
(\oR^n;O)$. This result is a version of the Paley-Wiener-Schwartz theorem
in the tempered ultrahyperfunction set-up.

\begin{theorem}[\cite{Daniel1}]
Let $f(z) \in \boldsymbol{{\mathscr H}^{*\,o}_c}$, where $C$ is an open convex cone.
Then the distribution $V \in H^\prime_{C^*}(\oR^n;O)$ has a uniquely
determined inverse Fourier-Laplace transform $f(z)=(2\pi)^{-n}
\bigl\langle V,e^{-i\langle \xi,z \rangle} \bigr\rangle$ which is holomorphic in
$T(C^\prime)$ and satisfies the estimate (\ref{Estimate1}), with $B[0;r]$ replaced by
$B(0;r)$.
\label{PWSTheo} 
\end{theorem}

The same proof as in Carmichael~\cite[Theorem 1, equation (4)]{Carmi2}
combined with the proofs of Theorems \ref{theorem1} and \ref{PWSTheo} shows that
the following corollary is true.

\begin{corollary}
Let $C$ be an open convex cone, and let $C^\prime \Subset C$.
Let $f(z) \in \boldsymbol{{\mathscr H}^{*\,o}_c}$. Then there exists
a unique element $V \in H^\prime_{C^*}(\oR^n;O)$ such that
\begin{equation}
{f(z)}={\mathscr F}^{-1}\bigl[e^{-\langle \xi,y \rangle}
V \bigr]\,\,,\quad z \in T(C^\prime;r)=\oR^n+i\bigl(C^\prime \setminus
\bigl(C^\prime \cap B(0;r)\bigr)\bigr)\,\,,
\label{equa12}
\end{equation}
where (\ref{equa12}) holds as an equality in ${\mathfrak H}^\prime(T(O))$.  
\label{corol2}
\end{corollary}

\begin{remark}
It is important to remark that in Theorems \ref{theorem1} and \ref{PWSTheo}
we are considering the inverse Fourier-Laplace transform $f(z)=(2\pi)^{-n}
\bigl\langle V,e^{-i\langle \xi,z \rangle} \bigr\rangle$, in opposition to the
Fourier-Laplace transform used in the proof of Theorem 1 of Ref.~\cite{Carmi2}.
In this case the proof of Corollary \ref{corol2} is achieved if we consider
$\xi$ as belon\-ging to the open half-space $\bigl\{\xi \in C^* \mid \langle \xi,y
\rangle < 0\bigr\}$, for $y \in C^\prime \setminus \bigl(C^\prime \cap B(0;r)\bigr)$,
since by hypothesis $f(z) \in \boldsymbol{{\mathscr H}^{*\,o}_c}$. Then,
from~\cite[Lemma 2, p.223]{Vlad} there is $\delta(C^\prime)$ such that for
$y \in C^\prime \setminus \bigl(C^\prime \cap B(0;r)\bigr)$ implies $\langle
\xi,y \rangle \leq -\delta(C^\prime) |\xi||y|$. This justifies the negative sign
in (\ref{equa12}).
\label{remark2}
\end{remark}

\section{Analytic Functionals in ${\mathfrak H}^\prime(T(O))$ Carried by the
Real Space}
\label{Sec4}
Let $\Omega$ be a closed set in $T(O)$. Let $\Omega_m$ be a closed neighborhood of $\Omega$
defined by
\[
\Omega_m=\bigl\{z \in \oC^n \mid {\rm dist}(z,\Omega) \leq 1/m \bigr\}\,\,.
\]
For a closed set $\Omega_m$ of $\oC^n$, ${\mathfrak H}_b(\Omega_m)$ is the space
of all continuous functions $\psi$ on $\Omega_m$ which are holomorphic in the interior
of $\Omega_m$ and satisfy
\[
\|\psi\|_{\Omega_m,N}=\sup_{{\substack{z \in \Omega_m \\ N \in \oN}}}
\bigl\{(1+|z|)^{N}|\psi(z)|\bigr\}\,\,.
\]
${\mathfrak H}_b(\Omega_m)$ is a Fr\'echet space with the seminorms $\|\psi\|_{\Omega_m,N}$.
If $m^\prime < m$, $\Omega_m \subset \Omega_{m^\prime}$, then we have the canonical injections
\begin{equation}
{\mathfrak H}_b(\Omega_{m^\prime}) \hookrightarrow {\mathfrak H}_b(\Omega_m)\,\,.
\label{canoinj21}
\end{equation}
We define the space ${\mathfrak H}(\Omega)$
\begin{equation}
{\mathfrak H}(\Omega)=\underset{m \rightarrow \infty}
{\lim {\rm proj}}\,\,{\mathfrak H}_b(\Omega_m)\,\,,
\label{limproj31}
\end{equation}
where the projective limit is taken following the restriction mappings
(\ref{canoinj21}). 

\begin{definition}
An analytic functional $U \in {\mathfrak H}^\prime(T(O))$ is carried by the
closed set $\Omega \subset \oC^n$ with respect to the decreasing sequence
$\{\Omega_m\}_{m=1}^\infty$ of neighborhoods of $\Omega$, if for every $m$ the
functional $U$ is already a functional on the space ${\mathfrak H}_b(\Omega_m)$ of
restrictions to $\Omega_m$ of functions in ${\mathfrak H}(T(O))$.
\end{definition}

In this section, in particular, we restrict ourselves to the case where
$\Omega$ is contained in $\oR^n=\bigl\{z \in \oC^n \mid z=x+iy, x\in \oR^n,
y=0\bigr\}$. In this case, every function $f(z) \in
\boldsymbol{{\mathscr H}^{*\,o}_c}$, which for each $y \in C^\prime$ as a
function of $x={\rm Re}\,z$ belongs to ${\mathfrak H}^\prime(T(O))$, is a
continuous linear functional on the space of {\bf restrictions} to $\oR^n$ of
functions in ${\mathfrak H}(T(O))$. Then, according to Theorem \ref{theorem1}({\it iii}),
$U=BV(f(z))$ the distributional boundary value of $f(z)$ is an element of
${\mathfrak H}^\prime(T(O))$ carried by $\oR^n$. 

Let $C$ be an open cone of the form $C=\bigcup_{j=1}^m C_j$, $m < \infty$,
where each $C_j$ is an proper open convex cone. If we write $C^\prime \Subset C$,
we mean $C^\prime=\bigcup_{j=1}^m C_j^\prime$ with $C_j^\prime \Subset C_j$.
Furthermore, we define by $C_j^*=\bigl\{\xi \in \oR^n \mid \langle \xi,x \rangle
\geq 0, \forall x \in C_j \bigr\}$ the dual cones of $C_j$, such that the dual
cones $C_j^*$, $j=1,\ldots,m$, have the properties
\begin{equation}
\oR^n \setminus \bigcup_{j=1}^m C_j^*\,\,,
\label{p1}
\end{equation}
and
\begin{equation}
C_j^* \bigcap C_k^*\,\,,j \not= k\,\,,j,k=1,\ldots,m\,\,,
\label{p2}
\end{equation}
are sets of Lesbegue measure zero. Assume that $V \in H^\prime_{C^*}(\oR^n;O)$
can be written as $V=\sum_{j=1}^{m} V_j$, where we define
\begin{equation}
V_j=D^\gamma_\xi[e^{h_K(\xi)}\lambda_j(\xi)g(\xi)]\,\,,
\label{Def2}
\end{equation}
with $\lambda_j(\xi)$ denoting the characteristic function of $C_j^*$, $j=1,\ldots,m$,
$g(\xi)$ being a bounded continuous function on $\oR^n$ and $h_K(\xi)=k|\xi|$ for a
convex compact set $K=\bigl[-k,k\bigr]^n$. In the following theorem not only
${\mathscr F}^{-1}[V]$ but also $V$ is represented as sum of boundary values of
holomorphic functions.

\begin{theorem}
For $V \in H^\prime_{C^*}(\oR^n;O)$ represented as $V=\sum_{j=1}^{m} V_j$
where
\[
V_j=D^\gamma_\xi[e^{h_K(\xi)}\lambda_j(\xi)g(\xi)]\,\,,
\]
with $\lambda_j(\xi)$ denoting the characteristic function of $C_j^*$, $j=1,\ldots,m$,
$g(\xi)$ being a bounded continuous function on $\oR^n$ and $h_K(\xi)=k|\xi|$ for a
convex compact set $K=\bigl[-k,k\bigr]^n$, the following statements are equivalent:
\newcounter{numero}
\setcounter{numero}{0}
\def\Prop{\addtocounter{numero}{1}\item[{$\boldsymbol{\sf St.{\thenumero}-}$}]}
\begin{enumerate}

\Prop ${\mathscr F}^{-1}[V] \in {\mathfrak H}^\prime(T(O))$ is carried by $\oR^n$.

\Prop Let $C$ be an open cone such that $C=\bigcup_{j=1}^m C_j$, where the $C_j$
are open convex cones such that (\ref{p1}) and (\ref{p2}) are satisfied.
$U={\mathscr F}^{-1}[V]$ is the sum of distributional boundary values in
${\mathfrak H}^\prime(T(O))$ of functions $f_j(z) \in \boldsymbol{{\mathscr H}^{*\,o}_{c_j}}$,
$j=1,\ldots,m$.

\Prop $V$ is the sum of distributional boundary values $V_j \in H^\prime_{C_j^*}(\oR^n;O)$
of functions $v_j(\zeta)$ holomorphic in $\oR^n+iC_j^*$, $j=1,\ldots,m$, satisfying for any
$C_j^{*\prime} \Subset C_j^*$, the estimate
\begin{equation}
\bigl|v_j(\zeta)\bigr| \leq {\boldsymbol{\sf K}_\varepsilon}(C_j^{*\prime})
(1+|\eta|)^{N} e^{k|\xi|}\,\,,\quad \eta \in C_j^{*\prime}\,\,.
\label{Estimate2} 
\end{equation}
\label{theorem2} 
\end{enumerate} 
\end{theorem} 

\begin{proof}
Proof that $\boldsymbol{\sf St.1} \Rightarrow \boldsymbol{\sf St.2}$.
Consider
\begin{equation}
h_y(\xi)=\int_{\oR^n}\frac{f(z)}{P(iz)}\,\,
e^{i\langle \xi,z \rangle}d^nx\,\,,\quad z \in T(C^\prime;r)\,\,, 
\label{eq41p}
\end{equation}
with $h_y(\xi)=e^{k|\xi|}g_y(\xi)$, where $g(\xi)$ is a bounded continuous
function on $\oR^n$, and $P(iz)=(-i)^{|\gamma|} z^\gamma$. By
hypothesis $f(z) \in \boldsymbol{{\mathscr H}^{*\,o}_c}$ and satisfies
(\ref{Estimate1}), with $B[0;r]$ replaced by $B(0;r)$. For this reason, for an
$n$-tuple $\gamma=(\gamma_1,\ldots,\gamma_n)$ of non-negative integers
conveniently chosen, we obtain
\begin{equation}
\Bigl|\frac{f(z)}{P(iz)}\Bigr|\leq
{\boldsymbol{\sf M}}(C^\prime)(1+|z|)^{-n-\varepsilon} e^{h_{c^*}(y)}\,\,,
\label{eq42} 
\end{equation}
where $n$ is the dimension and $\varepsilon$ is any fixed positive real
number. This implies that the function $h_y(\xi)$ exists and is a continuous function
of $\xi$. Further, by using arguments paralleling the analysis in~\cite[p.225]{Vlad}
and the Cauchy-Poincar\'e Theorem~\cite[p.198]{Vlad}, we can show that the
function $h_y(\xi)$ is independent of $y={\rm Im}\,z$. Therefore, we denote
the function $h_y(\xi)$ by $h(\xi)$.

From (\ref{eq42}) we have that $f(z)/P(iz) \in L^2$ as a function of
$x={\rm Re}\,z \in \oR^n$, $y \in C^\prime \setminus \bigl(C^\prime
\cap B(0;r)\bigr)$. Hence, from (\ref{eq41p}) and the Plancherel theorem
we have that $e^{-\langle \xi,y \rangle}h(\xi) \in L^2$ as a function of
$\xi \in \oR^n$, and  
\begin{equation}
\frac{f(z)}{P(iz)}={\mathscr F}^{-1}\bigl[e^{-\langle \xi,y \rangle}
h(\xi)\bigr](x)\,\,,\quad z \in T(C^\prime;r)\,\,,
\label{eq43} 
\end{equation}
where the inverse Fourier transform is in the $L^2$ sense. It should be noted
that for Eq.(\ref{eq43}) to be true $\xi$ must belong to the open half-space
$\bigl\{\xi \in C^* \mid \langle \xi,y \rangle < 0\bigr\}$, for
$y \in C^\prime \setminus \bigl(C^\prime \cap B(0;r)\bigr)$, since by hypothesis
$f(z) \in \boldsymbol{{\mathscr H}^{*\,o}_c}$ (see Remark \ref{remark2}).

From (\ref{PRel2}), we have 
\begin{align}
\langle {\mathscr F}^{-1}[V],\psi \rangle&=\langle V,
{\mathscr F}^{-1}[\psi] \rangle \nonumber\\[3mm]
&=\sum_{j=1}^{m} \Bigl \langle D^\gamma_\xi \bigl(\lambda_j(\xi)h(\xi)\bigr),
\frac{1}{(2\pi)^n} \int_{\oR^n} \psi(x)e^{-i \langle \xi,x \rangle} d^nx
\Bigr \rangle \nonumber\\[3mm]
&=\sum_{j=1}^{m} \Bigl \langle \lambda_j(\xi)h(\xi),
\frac{1}{(2\pi)^n} \int_{\oR^n} D^\gamma_\xi \Bigl(\psi(x)
e^{-i \langle \xi,x \rangle} \Bigr) d^nx \Bigr \rangle \label{eq49}\\[3mm]
&=\sum_{j=1}^{m} \Bigl \langle \lambda_j(\xi)h(\xi),
\frac{(-i)^{|\gamma|}}{(2\pi)^n} \int_{\oR^n} x^\gamma \psi(x) e^{-i \langle \xi,x \rangle}
d^nx \Bigr \rangle \nonumber\\[3mm]
&=\sum_{j=1}^{m} \lim_{C^\prime \ni y \rightarrow 0}\Bigl \langle
{\mathscr F}^{-1}\bigl[e^{-\langle \xi,y \rangle}
\lambda_j(\xi)h(\xi)\bigr],(-i)^{|\gamma|}(x+iy)^\gamma \psi(x) \Bigr \rangle \nonumber\,\,,
\end{align}
where we have used the fact that the differentiation under the integral sign is valid.
We note that $\psi(x) \in {\mathfrak H}(T(O))$ implies $(z^\gamma \psi(x)) \in
{\mathfrak H}(T(O))$ as a function of $x={\rm Re}\,z \in \oR^n$.

From (\ref{eq43}), we have for $z \in T(C^\prime;r)$
\begin{align}
\Bigl \langle i^{|\gamma|} (x+iy)^{-\gamma} f(x+iy), \psi(x) \Bigl \rangle&=
\Bigl \langle {\mathscr F}^{-1}\bigl[e^{-\langle \xi,y \rangle}
h(\xi)\bigr],\psi(x) \Bigl \rangle\,\,,\nonumber\\[3mm]
&=\sum_{j=1}^{m} \Bigl \langle {\mathscr F}^{-1}\bigl[e^{-\langle \xi,y \rangle}
\lambda_j(\xi)h(\xi)\bigr], \psi(x) \Bigr \rangle\,\,.
\label{eq410}
\end{align}
Combining (\ref{eq49}) and (\ref{eq410}), we obtain
\begin{align*}
\langle {\mathscr F}^{-1}[V],\psi \rangle=\sum_{j=1}^{m}
\lim_{C^\prime \ni y \rightarrow 0} \langle f_j(x+iy),\psi(x) \rangle
=\sum_{j=1}^{m} \langle U_j,\psi \rangle=\langle U,\psi \rangle\,\,.
\end{align*}
Thus the inverse Fourier transform ${\mathscr F}^{-1}[V]$ is a distributional boundary
value of $\sum_{j=1}^{m} f_j(z)$ in the sense of weak convergence. But
from~\cite[Corollary 1, p.358]{Treves} the latter implies strong convergence since
${\mathfrak H}(T(O))$ is Montel.

Proof that $\boldsymbol{\sf St.2} \Rightarrow \boldsymbol{\sf St.3}$. In
the following we shall write $\zeta_j=\xi+i\eta_j$, with $\eta_j \in C_j^{\prime *}
\Subset C_j^*$. It follows that for $\varphi \in H(\oR^n;O)$
\begin{align*}
\langle V, \varphi \rangle 
&=\sum_{j=1}^{m} \Bigl \langle U_j, \int_{\oR^n} \varphi(\xi)
e^{i \langle \xi,z \rangle} d^n\xi \Bigr \rangle \\[3mm]
&=\sum_{j=1}^{m} \lim_{C^{\prime *} \ni \eta_j \rightarrow 0} \Bigl \langle U_j,
\int_{\oR^n} \varphi(\xi)e^{i \langle \zeta,z \rangle} d^n\xi \Bigr \rangle \\[3mm]
&=\sum_{j=1}^{m} \lim_{C^{\prime *} \ni \eta_j \rightarrow 0} \int_{\oR^n}
\varphi(\xi) \langle U_j, e^{i \langle \zeta,z \rangle} \rangle\,d^n\xi \\[3mm]
&=\sum_{j=1}^{m} \lim_{C^{\prime *} \ni \eta_j \rightarrow 0} \int_{\oR^n}
\varphi(\xi) v_j(\zeta)\,d^n\xi\,\,.
\end{align*}
Each function $v_j(\zeta)=\langle U_j, e^{i \langle \zeta,z \rangle} \rangle$
is holomorphic in $\oR^n+iC_j^*$, $j=1,\ldots,m$.

Now, since we consider $V$ as a distribution in $H^\prime(\oR^n;O)$,
$U={\mathscr F}^{-1}[V]$ acts, in principle, on functions in ${\mathfrak H}(T(O))$. 
We need a decomposition of $U$ as sum of analytic functionals carried by closed convex
cones $C_j$ in $\oR^n$. For that purpose we introduce the following space,
${\mathfrak H}((C_j)_\varepsilon)$, of analytic functions. Let $C_j$ be a closed
convex cone in $\oR^n$. As in Ref.~\cite{BruNa1}, for $\varepsilon > 0$, we define
the closed complex $\varepsilon$-neighborhood of $C_j$ by
\[
(C_j)_\varepsilon=\bigl\{z \in \oC^n \mid {\boldsymbol{\exists}}\,x \in C_j,
|{\rm Re\,}z-x|+|{\rm Im\,}z|_\beta \leq \varepsilon \bigr\}\,\,,
\]
where $|y|_\beta$ is a norm of $\oR^n$ satisfying $|y|_\beta
\geq |y|$ for the Euclidean norm $|y|$. Let $L_\alpha$ be the closure of
$(C_j)_{\varepsilon/(1+1/\alpha)}$. ${\mathfrak H}_b(L_\alpha)$ is, by definition,
the space of all continuous functions $\psi$ on $L_\alpha$ which are holomorphic
in the interior of $L_\alpha$ and satisfy
\[
\|\psi\|_{L_\alpha,N}=\sup_{{\substack{z \in L_\alpha \\ N \in \oN}}}
\bigl\{(1+|z|)^N|\psi(z)|\bigr\}\,\,.
\]
${\mathfrak H}_b(L_\alpha)$ is a Fr\'echet space with the seminorms
$\|\psi\|_{L_\alpha,N}$. If $\alpha_{_1} < \alpha_{_2}$, $L_{\alpha_{_1}}
\subset L_{\alpha_{_2}}$, then we have the canonical injections
\begin{equation}
{\mathfrak H}_b(L_{\alpha_{_2}}) \hookrightarrow {\mathfrak H}_b(L_{\alpha_{_1}})\,\,.
\label{canoinj2}
\end{equation}
Then, we define the space ${\mathfrak H}((C_j)_\varepsilon)$
\begin{equation}
{\mathfrak H}((C_j)_\varepsilon)=\underset{\alpha \rightarrow \infty}
{\lim {\rm proj}}\,\,{\mathfrak H}_b(L_\alpha)\,\,,
\label{limproj3}
\end{equation}
where the projective limit is taken following the restriction mappings
(\ref{canoinj2}). According to Ref.~\cite[Theorem 2.13]{BruNa1},
${\mathfrak H}$ is dense in ${\mathfrak H}((C_j)_\varepsilon)$. This implies that
$e^{i \langle \zeta,z \rangle}$ as a function of $z$ can be approximated in
${\mathfrak H}((C_j)_\varepsilon)$ by functions in ${\mathfrak H}$. We now invoke
the Riesz's Representation Theorem. If $U_j$ is a continuous linear functional
on ${\mathfrak H}((C_j)_\varepsilon)$, that is,
\[
|\langle U_j, \psi \rangle| \leq {\boldsymbol{\sf M}_\varepsilon}
\sup_{{\substack{z \in (C_j)_\varepsilon \\ N \in \oN}}}
\bigl\{(1+|z|)^{N}|\psi(z)|\bigr\}\,\,,
\]
for a fixed ${\boldsymbol{\sf M}_\varepsilon}$ depending on $\varepsilon$ and all $\psi$,
then there exists a measure $\mu_j$ on $(C_j)_\varepsilon$ such that for all $\psi(z)$ in
${\mathfrak H}((C_j)_\varepsilon)$
\[
\langle U_j, \psi \rangle=\int_{(C_j)_\varepsilon} \psi(z)d\mu_j(z)\,\,.
\]
Moreover, the number $\int |d\mu_j(z)|$ may be taken as the bound of $U_j$, {\it i.e.},
the number ${\boldsymbol{\sf M}_\varepsilon}$ above. Thus, it follows that
\[
|\langle U_j, \psi \rangle| \leq \|\psi\|_{(C_j)_\varepsilon,N}
\int_{(C_j)_\varepsilon} \frac{|d\mu_j(z)|}{(1+|z|)^{N}}\,\,.
\]
The above representation yields for $U_j$ carried by $C_j$ in $\oR^n$
\[
\int_{(C_j)_\varepsilon} \frac{|d\mu_j(z)|}{(1+|z|)^{N}}
\leq
\int_{(C_j)_\varepsilon} \frac{|d\mu_j(z)|}{(1+|x|)^{N}}
\leq
{\boldsymbol{\sf M}_\varepsilon}\,\,.
\]

Furthermore, using the mean value theorem, one can verify that, for each $\varphi
\in H(\oR^n;O)$ and $\eta \in C_j^{\prime *}$, the Riemann sums corresponding to the
integral
\[
\int_{\oR^n} \varphi(\xi) e^{i \langle \zeta,z \rangle} d\xi
\]
converge in the space ${\mathfrak H}(C_j)$, $j=1,\ldots,m$ to $\psi_\eta(z)=
\psi(z)e^{-z\eta}$, where
\[
\psi(z)=\int_{\oR^n} \varphi(\xi) e^{i \langle \xi,z \rangle} d\xi\,\,.
\]
Hence, the identity
\[
\langle U_j,\psi(z)e^{-z\eta} \rangle=\int_{\oR^n} \varphi(\xi) v_j(\zeta)\,d^n\xi\,\,,
\]
holds in $C_j^\prime$. It is straightforward to prove the convergence $\psi_\eta
\rightarrow \psi$ as $ C_j^{\prime *} \Subset C_j^*\ni \eta \rightarrow 0$.
Finally, the estimate (\ref{Estimate2}) is a consequence of the inequality
$|\langle U_j, e^{i \langle \zeta,z \rangle} \rangle| \leq \|U_j\|_{(C_j)_\varepsilon,N}
\|e^{i \langle \zeta,z \rangle}\|_{(C_j)_\varepsilon,N}$. Then, it follows that
\begin{align*}
\bigl|v_j(\zeta)\bigr| \leq {\boldsymbol{\sf M}_\varepsilon}
\sup_{{\substack{z \in (C_j)_\varepsilon \\ N \in \oN}}}
(1+|x|)^{N} |e^{i \langle \zeta,z \rangle}|
={\boldsymbol{\sf M}_\varepsilon}
\sup_{{\substack{z \in (C_j)_\varepsilon \\ N \in \oN}}}
(1+|x|)^{N}e^{\langle \eta,x \rangle}e^{\langle \xi,y \rangle}. 
\end{align*}
Now, assume that $x$ belongs to the open half-space $\bigl\{x \in C_j \mid
\langle \eta,x \rangle < 0\bigr\}$. Then, for some fixed number $1 \geq \delta(C_j^{*\prime})
> 0$, it follows that $\langle \eta,x \rangle \leq -\delta(C_j^{*\prime})|\eta||x|$ for
$\eta \in C_j^{*\prime}$. Thus,
\begin{align*}
\bigl|v_j(\zeta)\bigr| &\leq {\boldsymbol{\sf M}_\varepsilon} e^{k|\xi|}
\sup_{{\substack{t \geq 0 \\ N \in \oN}}}
(1+|t|)^{N} e^{-\delta(C_j^{*\prime})t|y|}\\[3mm]
&\leq {\boldsymbol{\sf K}_\varepsilon}(C_j^{*\prime})
(1+|\eta|)^{-N} e^{k|\xi|}
\,\,,\quad
\eta \in C_j^{*\prime} \Subset C_j^*\,\,. 
\end{align*}

Proof that $\boldsymbol{\sf St.3} \Rightarrow \boldsymbol{\sf St.1}$. It is obvious.
\end{proof}

\begin{remark}
An analogous theorem to the Theorem \ref{theorem2} was obtained
by J.W. de Roever~\cite{Roever} to the space of analytic functionals in
$Z^\prime$.   
\end{remark}

\section{Uniqueness Theorem}
\label{Sec5}
We now state the main theorem of this paper.

\begin{theorem}[Uniqueness Theorem]
Let $V \in H^\prime(\oR^n;O)$ be a distribution of exponential growth whose support
is contained in some proper convex cone $C^*$. Then only the whole of $\oR^n$ can be
a carrier of $U={\mathscr F}^{-1}[V]$.
\label{UniqueTheo}
\end{theorem}

For our proof of the Theorem \ref{UniqueTheo} we need a lemma on the analytic wavefront
set of $V$, denoted by $WF_A(V)$. Since the distributions of exponential growth are
embedded in the space of distributions, we can use the general facts from distribution
theory to analyse the analytic wavefront set $WF_A$ of a distribution
$V \in H^\prime(\oR^n;O)$. The reader, who wants to obtain further insight on the
concept of analytic wavefront set of distributions, is referred to the H\"ormander's
textbook~\cite[Chapters 8 and 9]{Hor2}.

\begin{lemma}
If $V \in H^\prime(\oR^n;O)$ and $U={\mathscr F}^{-1}[V]$ is carried by a closed
cone $C$, then
\[
WF_A(V) \subset \oR^n \times C\,\,.
\]
\label{lemmaWF}
\end{lemma}

\begin{proof}
Let $\{C_j\}_{j \in L}$ be a finite covering of closed properly convex
cones of $C$. Decompose $U$ as follows:
\begin{equation}
U=\sum\,U_j\,\,,
\label{decomp}  
\end{equation}
The decomposition (\ref{decomp}) will induce a representation of
$V$ in the form of a sum of boundary values of functions $v_j(\zeta)$,
such that $v_j(\zeta) \rightarrow V_j$ as $\eta \rightarrow 0$, $\eta
\in C_j^{*\prime} \subset C_j^*$. According to Theorem \ref{theorem2}, the family
of functions $v_j(\zeta)$ satisfy the estimate
\begin{align*}
\bigl|v_j(\zeta)\bigr| \leq {\boldsymbol{\sf K}_\varepsilon}(C_j^{*\prime})
(1+|\eta|)^{N} e^{k|\xi|}
\,\,,\quad
\eta \in C_j^{*\prime} \Subset C_j^*\,\,, 
\end{align*}
unless $\langle \eta,x \rangle \geq 0$ for $\eta \in C_j^*$ and $x \in C_j^\prime$.
This implies that the cones of ``bad'' directions responsible for the singularities
of these boundary values are contained in the dual cones of the base cones. So, we
have the inclusion
\begin{equation}
WF_A(u) \subset \oR^n \times \bigcup_j\,C_j\,\,.
\label{inclusion}  
\end{equation}
Then, by making a refinement of the covering and shrinking it to $C$,
we obtain the desired result.
\end{proof}

We now turn to the proof of Theorem \ref{UniqueTheo}. It is essentially the restriction
of proof of Soloviev's Uniqueness Theorem~\cite{Solo1,Solo2} to the space of distributions
of exponential growth, since the latter is embedded in the space of distributions.
For this reason, we limit ourselves to explain, exactly as in~\cite{Solo1,Solo2}, the role
that the Lemma \ref{lemmaWF} plays in the derivation of Theorem \ref{UniqueTheo}. We begin with
the simplest case when $0 \in \supp\,V$.  Then every vector in the cone $-C \setminus \{0\}$
is an external normal to the support at the point 0.  By Theorem 9.6.6 of~\cite{Hor2}, all
the nonzero elements of the linear span of external normals belong to $WF_A(V)_{\xi=0}$.
Because the cone $C^*$ is properly convex, the interior of $C$ is not
empty, and this linear span covers $\oR^n$. Therefore, by Lemma \ref{lemmaWF}, each carrier
cone of $U={\mathscr F}^{-1}[V]$ must then coincide with $\oR^n$. The general case can be
reduced to this special case by considering the series $\sum^\infty_{j=1}a_j V_j$,
(of suitable contractions), where $V_j(\xi)=V(j \xi)$.
As shown in~\cite{Solo1,Solo2}, the coefficients $a_j$ can be chosen such
that this series converges in ${\mathscr D}^{\prime}(\oR^n)$ to a distribution whose
support contains the point 0 and whose Fourier transform is carried by the
same cones that $U={\mathscr F}^{-1}[V]$ is. Then the proof of Theorem \ref{UniqueTheo}
follows from the fact that we can consider the distributions $V \in {\mathscr D}^{\prime}(\oR^n)$
which belong to $H^\prime(\oR^n;O)$ in accordance to Lemma \ref{lemmaMari}.

\section{Connection with Field-Theoretical Models with a Fundamental Length}
\label{Sec6}
This section represents a border between mathematics and physics. The results given
here have an independent, purely mathematical, interest. It is to be hoped that
the results of this section meet with the interest of theoretical physicists and
mathematicians who are working with quantum field theory, or string theory. Recent
de\-ve\-lo\-pments~\cite{BruNa1} have shown the need for analytic functionals which
are Fourier transform of distributions of exponential growth in order to treat quantum
field theories which require a fundamental length, as indicate from string theory.
 
According to Wightman~\cite{SW}-\cite{Haag}, the conventional postulates of QFT can
be fully reexpressed in terms of an equivalent set of properties of the vacuum expectation
values of their ordinary field products, called Wightman distributions
\begin{eqnarray}
{\mathfrak W}_m(f_{1} \otimes \cdots \otimes f_{m})\overset{\text{\rm def}}
{=}\langle\Omega_o \mid \Phi(f_1) \cdots \Phi(f_m) \mid \Omega_o\rangle\,\,,
\label{vev}
\end{eqnarray}
where $(f_{1} \otimes \cdots \otimes f_{m})= f_1(x_1) \cdots f_m(x_m)$
is considered as an element of ${\mathscr S}(\oR^{4m})$, and
$\mid \Omega_o\rangle$ is the vacuum vector, unique vector time-translation
invariant of the Hilbert space of states.

\begin{remark}
To keep things as simple as possible, we will assume that the Wightman
distributions are ``functions'' ${\mathfrak W}_m(x_{1},\ldots,x_{m})$.
The reader can easily supply the necessary test functions.
\end{remark}

As a general rule, the continuous
linear functionals ${\mathfrak W_m}(x_{1},\ldots,x_{m})$ are assumed to satisfy
the following axioms:

\setcounter{numero}{0}
\def\Prop{\addtocounter{numero}{1}\item[{$\boldsymbol{\sf Ax.{\thenumero}}$}]}
\begin{enumerate}

\Prop (Temperedness). The Wightman functions
${\mathfrak W_m}(x_{1},\ldots,x_{m})$ are tempered distributions in
${\mathscr S}^\prime(\oR^{4m})$, for all $m \geq 1$. This property is included
in the list of properties for a QFT for technical reasons.

\bigskip

\Prop (Poincar\'e Invariance). Wightman functions are invariant under the
Poincar\'e group
\[
{\mathfrak W_m}(\Lambda x_{1}+a,\ldots,\Lambda x_{m}+a)=
{\mathfrak W_m}(x_{1},\ldots,x_{m})\,\,.
\]

\bigskip

\Prop (Spectral Condition). The Fourier transforms of the Wightman functions
have support in the region
\begin{gather*}
\Bigm\{(p_1,\ldots,p_m)\in {\Bbb R}^{4m}\,\,\bigm|\,\,\sum_{j=1}^{m}p_j=0,\,\,
\sum_{j=1}^{k}p_j \in {\overline V}_+,\,\,k=1,\dots,m-1 \Bigm\}\,\,,
\end{gather*}
where ${\overline V}_+=\{(p^0,{\boldsymbol p}) \in \oR^4 \mid p^2 \geq 0,
p^0 \geq 0\}$ is the closed forward light cone.

\bigskip

\Prop (Local commutativity). This property has origin in the quantum
principle that operator observables $\Phi(x)$ corresponding to independent
measurements must comute. 
\[
{\mathfrak W}_m(x_{1},\ldots,x_j,x_{j+1},\ldots,x_{m})=
{\mathfrak W}_m(x_{1},\ldots,x_{j+1},x_j,\ldots,x_{m})\,\,,
\]
if $(x_j-x_{j+1})^2<0$.

\bigskip
 
\Prop For any finite set $f_o,f_1,\ldots,f_N$ of test functions such that
$f_o \in \oC$, $f_j \in {\mathscr S}(\oR^{4j})$ for $1 \leq j \leq N$,
one has
\begin{align*}
\sum_{k,\ell=0}^N {\mathfrak W}_{k+\ell}(f_k^* \otimes f_\ell)
\geq 0\,\,.
\end{align*}

\bigskip

\Prop (Hermiticity). A neutral scalar field must be real valued. This implies that
\[
{\mathfrak W}_m(x_{1},x_2,\ldots,x_{m-1},x_{m})=
\overline{{\mathfrak W}_m(x_{m},x_{m-1},\ldots,x_{1},x_2})\,\,.
\]

\end{enumerate} 

In string theory, it is said that there is a fundamental length $\ell > 0$ such that one
cannot distinguish events which occur in a smaller distance than $\ell$~\cite{Pol}. Therefore,
string theory is non-localizable. Hence, generalizing the properties $\boldsymbol{\sf Ax.1}$
to $\boldsymbol{\sf Ax.6}$ in string theory is not as simple. Here, the question is:
how can the Property $\boldsymbol{\sf Ax.4}$ be described in field theory with a
fundamental length? For this question, one answer has been suggested by
Br\"uning-Nagamachi~\cite{BruNa1}. They have conjectured that tempered ultrahyperfunctions
are well adapted for their use in quantum field theory with a fundamental length. Although
tempered ultrahyperfunctions have no standard localization properties, a model for
relativistic quantum field theory with a fundamental length can be constructed which offers
many familiar features. The analysis of Br\"uning-Nagamachi~\cite{BruNa1} has shown that the
vacuum expectation values of a QFT with a fundamental length in terms of tempered
ultrahyperfunctions satisfies a number of specific properties. We shall not give all axioms
defining a quantized field with a fundamental length but only those which are needed in this
section (we refer to Br\"uning-Nagamachi~\cite{BruNa1} for details).

\setcounter{numero}{0}
\def\Prop{\addtocounter{numero}{1}\item[{$\boldsymbol{\sf Ax.^\prime{\thenumero}}$}]}
\begin{enumerate}

\Prop ${\mathfrak W_0}=1$, ${\mathfrak W_m} \in {\mathscr U}_c(\oR^{4m})$
for $n \geq 1$, and ${\mathfrak W_m}(f^*)=\overline{{\mathfrak W_m}(f)}$,
for all $f \in {\mathfrak H}(T(\oR^{4m}))$, where $f^*(z_1,\ldots,z_m)=
\overline{f(\bar{z}_1,\ldots,\bar{z}_m)}$.

\bigskip

\Prop The Wightman functionals ${\mathfrak W_m}$ are invariant under
the Poincar\'e group
\[
{\mathfrak W_m}(\Lambda x_{1}+a,\ldots,\Lambda x_{m}+a)=
{\mathfrak W_m}(x_{1},\ldots,x_{m})\,\,.
\]

\bigskip

\Prop Spectral condition. Since the Fourier transformation of tempered
ultrahyperfunctions are distributions, the spectral condition is not
so much different from that of Schwartz distributions. Thus, for every
$m \in \oN$, there is $\widehat{{\mathfrak W}}_{m} \in
H_{V^*}^\prime(\oR^{4m},\oR^{4m})$~\cite{BruNa1}, where
\begin{equation}
H^\prime_{V^*}(\oR^{4m},\oR^{4m})=\Bigl\{V \in H^\prime(\oR^{4m},\oR^{4m}) \mid
\supp\,(\widehat{{\mathfrak W}}_{m}) \subset V^* \Bigr\}\,\,, 
\label{EQ31'} 
\end{equation}
with $V^*$ being the properly convex cone defined by
\begin{gather*}
\Bigm\{(p_1,\ldots,p_m)\in {\Bbb R}^{4m}\,\,\bigm|\,\,\sum_{j=1}^{m}p_j=0,\,\,
\sum_{j=1}^{k}p_j \in {\overline V}_+,\,\,k=1,\dots,m-1 \Bigm\}\,\,,
\end{gather*}
where ${\overline V}_+=\{(p^0,{\boldsymbol p}) \in \oR^4 \mid p^2 \geq 0,
p^0 \geq 0\}$ is the closed forward light cone.

\bigskip

\Prop Extended local commutativity condition. Let $f,g$ be two test functions
in ${\mathfrak H}(T(\oR^4))$, then the fields $\Phi(f)$ and $\Phi(g)$ are said
to commute for any relative spatial separation $\ell^\prime > \ell$ of
their arguments, if the functional
\begin{align}
{\boldsymbol{\sf F}}=\bigl\langle\Theta \mid \bigl[\varphi(f),\varphi(g)\bigr]
\mid \Psi\bigr\rangle
=\bigl\langle\Theta \mid \bigl(\varphi(f) \varphi(g)-
\varphi(g) \varphi(f)\bigr)\mid \Psi\bigr\rangle\,\,,
\label{AofAC}
\end{align}
is carried by the set $M^{\ell^\prime}=\bigl\{\bigr(z_1,z_2)
\in \oC^{8} \mid z_1-z_2 \in V^{\ell^\prime}\}$, for any vectors
$\Theta,\Psi\in D_0$, {\it i.e.}, if the functional ${\boldsymbol{\sf F}}$ can be
extended to a continuous linear functional on ${\mathfrak H}(M^{\ell^\prime})$.
Here, $V^{\ell}$ denotes the complex $\ell$-neighborhood of the light cone $V_+$
\[
V^{\ell}=\Bigl\{z \in \oC^4 \mid
\exists\,\,x \in V_+, |{\rm Re}\,z - x| +
|{\rm Im}\,z|_1 < \ell \Bigr\}\,\,.
\]
\end{enumerate} 

The remaining of this paper deals with the proof of some important theorems in 
a quantum field theory, namely the proofs of the CPT theorem and the theorem on
the Spin-Statistics connection in the setting of a quantum field theory with a fundamental
length. The proof of these results as given in the literature~\cite{SW}-\cite{Haag} usually
seem to rely on the local character of the distributions in an essential way. In the
approach which we follow the apparent source of difficulties in proving these results
is the fact that for functionals belonging to the space of tempered ultrahyperfunctions
the standard notion of the localization principle breaks down.

For simplicity we shall discuss the case of a scalar field. Let $\Phi$ be a Hermitian
scalar field. For this field, it is well-known that in terms of the Wightman functions,
a necessary and sufficient condition for the existence of CPT theorem is given by:
\begin{equation}
{\mathfrak W}_m(x_1,\ldots,x_m)={\mathfrak W}_m(-x_m,\ldots,-x_1)\,\,.
\label{cptt1}
\end{equation}
Under the usual temperedness assumption, the proof of the equality (\ref{cptt1}) as given
by Jost~\cite{Jost} starts of the weak local commutativity (WLC) condition, namely under
the condition that the vacuum expectation value of the commutator of $n$ scalar fields
vanishes outside the light cone, which in terms of Wightman functions takes the form
\begin{equation}
{\mathfrak W}_m(x_1,\ldots,x_m)-{\mathfrak W}_m(x_m,\ldots,x_1)=0\,\,,\quad
{\mbox{for}}\quad x_j-x_{j+1} \in {\mathscr J}_m\,\,.
\label{wlc}
\end{equation}
Jost's proof that the WLC condition (\ref{wlc}) is equivalent to the CPT symmetry
(\ref{cptt1}) one relies on the fact that the proper complex Lorentz group contains
the total spacetime inversion. Therefore, the equality (\ref{cptt1}) holds, taking in
account the symmetry property ${\mathscr J}_m=-{\mathscr J}_m$ in whole extended analyticity
domain, by the Bargman-Hall-Wightman (BHW) theorem. In particular, the BHW theorem has been
shown~\cite{BruNa1} to be applicable to domains of the form ${\mathscr T}_{m-1}=\oR^{4(m-1)}+
V_+(\ell^\prime_{1},\ldots,\ell^\prime_{{m-1}})$. Then, the Wightman
functions depending on the relative coordinates $W_m(\zeta_1,\ldots,\zeta_{m-1})$
can be extended to be a holomorphic function on the extended tube
\[
{\mathscr T}^{\rm ext.}_{m-1}=\Bigl\{(\Lambda\zeta_1,\ldots,\Lambda\zeta_{m-1}))
\mid (\zeta_1,\ldots,\zeta_{m-1}) \in {\mathscr T}_{m-1}, \Lambda \in
{\mathscr L}_+(\oC)\Bigr\}\,\,,
\]
which contains certain real points of type of the Jost points. 

In order to prove that CPT theorem holds in QFT with a fundamental length in terms
of tempered ultrahyperfunctions, an analogous of the WLC condition is now formulated:

\begin{definition}
The quantum field $\Phi$ defined on the test function space
${\mathfrak H}(T(\oR^4))$ is said to satisfy the weak extended local commutativity (WELC)
condition if the functional
\[
{\boldsymbol{\sf F}}={\mathfrak W}_m(z_1,\ldots,z_m)-
{\mathfrak W}_m(z_n,\ldots,z_1)\,\,,
\]
is carried by set $M_j^{\ell^\prime}=\Bigl\{\bigr(z_1,\ldots,z_m)
\in \oC^{4m} \mid z_j-z_{j+1} \in V^{\ell^\prime} \Bigr\}$.
\label{WELC}
\end{definition}

The WELC condition takes the form $W_m(\zeta_1,\ldots,\zeta_{m-1})-
W_{m}(-\zeta_{m-1},\ldots,-\zeta_1)$ in terms of the Wightman functions depending
on the relative coordinates $\zeta_j=z_j-z_{j+1} \in V^{\ell^\prime}$.

\begin{proposition}[\cite{DZH}]
In a quantum field theory defined on the test function space
${\mathfrak H}(T(\oR^4))$, the Wightman functions $W_m(\zeta_1,\ldots,\zeta_{m-1})$
and $W_{m}(-\zeta_{m-1},\ldots,-\zeta_1)$ satisfy the following equality
$W_m(\zeta_1,\ldots,\zeta_{m-1})=W_{m}(-\zeta_{m-1},\ldots,-\zeta_1)$
on their res\-pective domains of holomorphy.
\label{sym}
\end{proposition}

We now are in a position to state the main results of this section.
Combining the spectral condition (\ref{EQ31'}) with the Lemma \ref{lemmaMari} and
the theorems established in Sections \ref{Sec4} and \ref{Sec5} of the present paper,
we can proceed as in Soloviev~\cite[CPT Theorem]{Solo1} and \cite[Spin-Statistics Theorem]{Solo2}
--- by replacing the reference to the spaces $S^0$ (or $Z$) and $S^{\prime 0}$ (or $Z^\prime$)
by a reference to the spaces ${\mathfrak H}$ and ${\mathfrak H}^\prime$, and considering
that ${\mathscr S}^\prime \subset H^\prime \subset {\mathscr D}^\prime$ and
${\mathscr S}^\prime \subset {\mathfrak H}^\prime \subset Z^\prime$, with the
injections being continuous --- in order to show that the following theorems are true
to a quantum field theory defined on the test function space ${\mathfrak H}(T(\oR^4))$.

\begin{theorem}[CPT Theorem]
In order to a quantum field theory defined on the test function space
${\mathfrak H}(T(\oR^4))$ to be invariant under the CPT-operation is
necessary and sufficient that the WELC condition is fulfilled.
\label{cpttheo}
\end{theorem}

\begin{theorem}[Spin-Statistics Theorem]
Suppose that $\Phi$ and its Hermitian conjugate $\Phi^*$ satisfy the WELC condition
with the ``wrong'' connection of spin and statistics. Then $\Phi(x)\Omega_o=
\Phi^*(x)\Omega_o=0$.
\label{sstheo}
\end{theorem}

\section*{Acknowledgments}
The author would like to express his gratitude to Afr\^anio R. Pereira,
Winder A.M. Melo and to the Departament of Physics of the Universidade Federal
de Vi\c cosa (UFV) for the opportunity of serving as Visiting Researcher.



\begin{thebibliography}{99}

\bibitem{Solo1} M.A. Soloviev, ``{\em A uniqueness theorem for distributions and its
application to nonlocal quantum field theory,}'' {\bf J.Math.Phys.} {\bf 39} (1998)
2635.

\bibitem{Solo2} M.A. Soloviev, ``{\em Extension of the spin-statistics theorem
to nonlocal fields,}'' {\bf JETP Letters} {\bf 67} (1998) 621.

\bibitem{GelShi} I.M. Gelfand and G.E. Shilov, ``{\em Generalized functions,}'' Vol.2,
Academic Press Inc., New York, 1968.

\bibitem{Tiao1} J. Sebasti\~ao e Silva, ``{\em Les fonctions analytiques
commes ultra-distributions dans le calcul op\'erationnel,}'' {\bf Math. Ann.}
{\bf 136} (1958) 58.

\bibitem{Tiao2} J. Sebasti\~ao e Silva, ``{\em Les s\'eries de multip\^oles
des physiciens et la th\'eorie des ultradistributions,}'' {\bf Math. Ann.}
{\bf 174} (1967) 109.

\bibitem{Hasumi} M. Hasumi, ``{\em Note on the $n$-dimensional
ultradistributions,}'' {\bf T\^ohoku Math. J.} {\bf 13} (1961) 94.

\bibitem{Zie} Z. Ziele\'zny, ``{\em On the space of convolution operators
in ${\mathscr K}_1^\prime$,}'' {\bf Studia Math.} {\bf 31} (1968) 111.

\bibitem{Mari1} M. Morimoto, ``{\em Theory of tempered ultrahyperfunctions I, II,}''
{\bf Proc. Japan Acad.} {\bf 51} (1975) 87, 213.

\bibitem{Mari3} M. Morimoto, ``{\em Convolutors for ultrahyperfunctions,}''
Lecture Notes in Physics, vol.39, Springer-Verlag, 1975, p.49.

\bibitem{Carmi1} R.D. Carmichael, ``{\em Distributions of exponential growth
and their Fourier transforms,}'' {\bf Duke Math. J.} {\bf 40} (1973) 765.

\bibitem{Carmi2} R.D. Carmichael, ``{\em The tempered ultra-distributions of
J. Sebasti\~ao e Silva,}'' {\bf Portugaliae Mathematica}
{\bf 36} (1977) 119.

\bibitem{Carmi3} R.D. Carmichael, ``{\em Distributional boundary values and
the tempered ultra-distributions,}'' {\bf Rend. Sem. Mat. Univ. Padova}
{\bf 56} (1977) 101.

\bibitem{SPinto1} J.S. Pinto, ``{\em Silva tempered ultradistributions,}''
{\bf Portugaliae Mathematica} {\bf 47} (1990) 267.

\bibitem{Suwa} M. Suwa, ``{\em Distributions of exponential growth with
support in a proper convex cone,}'' {\bf Publ. RIMS, Kyoto Univ.} {\bf 40} (2004) 565.

\bibitem{BruNa1} E. Br\"uning and S. Nagamachi, ``{\em Relativistic quantum field
theory with a fundamental length,}'' {\bf J. Math. Phys.} {\bf 45} (2004) 2199.

\bibitem{Andreas} A.U. Schmidt, ``{\em Asymptotic hyperfunctions, tempered
hyperfunctions, and asymptotic expansions,}'' {\bf Int. J. Math. Math. Sc.}
{\bf 5} (2005) 755.

\bibitem{DanHenri} D.H.T. Franco and L.H. Renoldi, {\em A note on Fourier-Laplace
transform and analytic wave front set in theory of tempered ultrahyperfunctions,}''
{\bf J. Math. Anal. Appl.} {\bf 325} (2007) 819.

\bibitem{Daniel1} D.H.T. Franco, ``{\em The edge of the wedge theorem for tempered
ultrahyperfunctions,}'' preprint CEFT-SFM-DHTF06/1, Revised Version.

\bibitem{DZH} D.H.T. Franco, J.A. Louren\c co and L.H. Renoldi,
``{\em Ultrahyperfunctional approach to non-commutative quantum field theory,}''
preprint CEFT-SFM-DHTF06/2, Revised Version.

\bibitem{Daniel2} D.H.T. Franco, ``{\em The edge of the wedge theorem for
tempered ultrahyperfunctions II. A generalized version,}'' preprint CEFT-SFM-DHTF07/1.

\bibitem{Roever} J.W. de Roever, ``{\em Analytic representations and Fourier
transforms of analytic functionals in $Z^\prime$ carried by the real space,}'' 
{\bf SIAM J. Math. Anal.} {\bf 9} (1978) 996.

\bibitem{Vlad} V.S. Vladimirov, ``{\em Methods of the theory of
functions of several complex variables,}'' M.I.T. Press, Cambridge, Mass., 1966.

\bibitem{Treves} F. Treves, ``{\em Topological vector spaces, distributions
and kernels,}'' Academic Press, 1967.

\bibitem{Hor2} L. H\"ormander, ``{\em The analysis of linear partial
differential operators I,}'' Springer Verlag, Second Edition, 1990.

\bibitem{SW} R.F. Streater and A.S. Wightman, ``{\em PCT, spin and statistics, and
all that,}'' Addison--Wesley, Redwood City, 1989.

\bibitem{Jost1} R. Jost, ``{\em The general theory of quantized fields,}''
Providence, AMS, 1965. 

\bibitem{BLOT} N.N. Bogoliubov, A.A. Logunov,  A.I. Oksak and
I.T. Todorov, ``{\em General principles of quantum field theory,}'' Kluwer, Dordrecht,
1990.
 
\bibitem{Haag} R. Haag, ``{\em Local quantum physics: fields, particles and algebras,}''
Second Revised Edition, Springer, 1996.

\bibitem{Pol} J. Polchinski, ``{\em String theory I,}'' Cambridge Universty Press,
Cambridge, 1998.

\bibitem{Jost} R. Jost, ``{\em Eine bemerkung zum CPT,}'' {\bf Helv. Phys. Acta} {\bf 30}
(1957) 409.

\end{thebibliography}
\end{document}